\documentclass[10pt,a4paper]{amsart}

\usepackage{amssymb}
\usepackage{amsmath}
\usepackage{amscd}

\begin{document}

\theoremstyle{plain}
\newtheorem{theorem}{Theorem}[section]
\newtheorem{proposition}[theorem]{Proposition}
\newtheorem{lemma}[theorem]{Lemma}
\newtheorem{corollary}[theorem]{Corollary}
\newtheorem{conj}[theorem]{Conjecture}

\theoremstyle{definition}
\newtheorem{definition}[theorem]{Definition}
\newtheorem{exam}[theorem]{Example}
\newtheorem{remark}[theorem]{Remark}

\numberwithin{equation}{section}

\title[$m$-quasi-Einstein metrics on quadratic Lie groups]
{Non-trivial $m$-quasi-Einstein metrics on quadratic Lie groups}

\author{Zhiqi Chen}
\address{School of Mathematical Sciences and LPMC \\ Nankai University \\ Tianjin 300071, P.R. China} \email{chenzhiqi@nankai.edu.cn}

\author{Ke Liang}
\address{School of Mathematical Sciences and LPMC \\ Nankai University \\ Tianjin 300071, P.R. China}\email{liangke@nankai.edu.cn}

\author{Fahuai Yi}
\address{School of Mathematical Sciences\\ South China Normal University\\ Guangzhou 510631, P.R. China}\email{fhyi@scnu.edu.cn}



\begin{abstract}
We call a metric $m$-quasi-Einstein if $Ric_X^m$ (a modification of the $m$-Bakry-Emery Ricci tensor in terms of a suitable vector field $X$) is a constant multiple of the metric tensor. It is a generalization of Einstein
metrics which contains Ricci solitons. In this paper, we focus on left-invariant vector fields and left-invariant Riemannian metrics on quadratic Lie groups. First we prove that any left-invariant vector field $X$ such that the left-invariant Riemannian metric
on a quadratic Lie group is $m$-quasi-Einstein is a Killing field. Then we construct infinitely many non-trivial $m$-quasi-Einstein metrics on solvable quadratic Lie groups $G(n)$ for $m$ finite.
\end{abstract}

\maketitle


\setcounter{section}{0}
\section{Introduction}
A natural extension of the Ricci tensor is the $m$-Bakry-Emery Ricci tensor
\begin{equation}
\mathrm{Ric}^m_f=\mathrm{Ric}+\nabla^2f-\frac{1}{m}df\otimes df
\end{equation}
where $0<m\leqslant \infty$, $f$ is a smooth function on $M^n$, and $\nabla^2f$ stands for the Hessian form. Instead of a gradient of a smooth function $f$ by a vector
field $X$, $m$-Bakry-Emery Ricci tensor was extended by Barros and Ribeiro Jr in \cite{BR1} and
Limoncu in \cite{L1} for an arbitrary vector field $X$ on $M^n$ as follows:
\begin{equation}
\mathrm{Ric}^m_X=\mathrm{Ric}+\frac{1}{2}{\mathfrak L}_Xg-\frac{1}{m}X^*\otimes X^*
\end{equation}
where ${\mathfrak L}_Xg$ denotes the Lie derivative on $M^n$ and $X^*$ denotes the canonical 1-form
associated to $X$. With this setting $(M^n,g)$ is called an $m$-quasi-Einstein metric, if there exist a vector
field $X\in {\mathfrak X}(M^n)$ and constants $m$ and $\lambda$ such that
\begin{equation}\label{qEin}
\mathrm{Ric}^m_X=\lambda g.\end{equation}

An $m$-quasi-Einstein metric is called trivial when $X\equiv 0$. The triviality definition is equivalent to say that $M^n$ is an Einstein manifold.
When $m=\infty$, the equation (1.3) reduces to a Ricci soliton, for more details see \cite{Ca1} and the references therein. If $m$ is a positive integer and $X$ is a gradient vector field, the condition corresponds to a warped product Einstein metric, for more details see \cite{HPW1}. Classically the study on $m$-quasi-Einstein are considered when $X$ is a gradient of a smooth function $f$ on $M^n$, see \cite{An1,AK1,CSW1,Cor1,ELM1,KK1}.

In this paper, we focus on left-invariant Riemannian metrics on quadratic Lie groups which include compact Lie groups and semisimple Lie groups as special classes. First, we have the following theorem.
\begin{theorem}\label{killingfield}
Let $G$ be a quadratic Lie group with a left-invariant Riemannian metric $\langle\cdot,\cdot\rangle$ and the Lie algebra $\mathfrak g$ of left-invariant vector fields. If $X\in {\mathfrak g}$ and $\mathrm{Ric}^m_X=\lambda\langle\cdot,\cdot\rangle$, then $X$ is a Killing field.
\end{theorem}

Theorem 1.1 for compact Lie groups is proved in \cite{CZ13}, and there are non-trivial
$m$-quasi-Einstein metrics on homogeneous manifolds \cite{BRJ1,CZ13}. Moreover, compact homogeneous Ricci solitons
are Einstein, which is equivalent with the conclusion that the vector field is a Killing field, by the
work of Petersen-Wylie \cite{PW1} and Perelman \cite{Pere1}. In \cite{Jab1}, Jablonski gives a
new proof. Essentially, Theorem 1.1 for $m$ infinite holds for semisimple Lie groups
by Jablonski's proof.

Next we pay attention to solvable quadratic Lie groups.
Here we study a class of simply connected solvable quadratic Lie groups $G(n)$ for $n\geq 1$, the derived algebras of whose Lie algebras are Heisenberg Lie algebras of dimension $2n+1$, and prove
\begin{theorem}\label{G(n)}
Every solvable quadratic Lie groups $G(n)$ admits infinitely many non-equivalent non-trivial $m$-quasi-Einstein metrics for $m$ finite.
\end{theorem}

In order to prove Theorem~\ref{G(n)}, we first obtain a formula of the Ricci curvature with respect to a left-invariant Riemannian metric which holds for any quadratic Lie group. It is a natural extension of the formula on compact semisimple Lie groups, which is given by Sagle \cite{Sa1} and simpler proved by D'Atri and Ziller in \cite{DZ1}. Based on it, we get a computable formula of the Ricci curvature on quadratic Lie groups, i.e. Theorem~\ref{prop}, which is the fundament of the proof for Theorem~\ref{G(n)}.

\section{The proof of Theorem~\ref{killingfield}}\label{section2}
A quadratic Lie algebra is a Lie algebra $\mathfrak g$ together with a pseudo-Riemannian metric $(\cdot,\cdot)\colon \mathfrak{g}\otimes\mathfrak{g}\to \mathbb{R}$ that is invariant under the adjoint action, i.e.
$$([X,Y],Z)+(Y,[X,Z])=0 \text{ for any } X,Y,Z\in {\mathfrak g}.$$
A Lie group $G$ is called a quadratic Lie group if and only if the Lie algebra $\mathfrak g$ of left-invariant vector fields is a quadratic Lie algebra. It is easy to see that $\mathrm{tr} \mathrm{ad} X=0$ for any $X\in {\mathfrak g}$, i.e. a quadratic Lie group is unimodular.

Let $M$ denote the set of left-invariant Riemannian metrics on a unimodular Lie group $G$. For any left-invariant Riemannian metric $Q$ on $G$, the tangent space $T_QM$ at $Q$ is the set of left-invariant symmetric, bilinear forms on $\mathfrak g$. Define a Riemannian metric on $M$ by
$$(v,w)_Q=\mathrm{tr}\ vw=\sum_iv(e_i, e_i)w(e_i, e_i),$$
where $v,w\in T_QM$ and $\{e_i\}$ is a $Q$-orthonormal basis of $\mathfrak g$.

Let $G$ be a unimodular Lie group with a left-invariant Riemannian metric $\langle\cdot,\cdot\rangle$. Given $Q\in M$, denote by $ric_Q$
and $sc_Q$ the Ricci and scalar curvatures of $(G,Q)$, respectively. The
gradient of the function $sc: M\rightarrow{\mathbb R}$ is
\begin{equation}\label{sc}(\mathrm{grad}\ sc)_Q =-ric_Q\end{equation}
relative to the above Riemannian metric on $M$, see \cite{He1} or \cite{Ni1}.

Assume that $\langle\cdot,\cdot\rangle$ is $m$-quasi-Einstein. That is, there exists a vector field $X$ on $G$ such that $\mathrm{Ric}+\frac{1}{2}{\mathfrak L}_X\langle,\rangle-\frac{1}{m}X^*\otimes X^*=\lambda\langle\cdot,\cdot\rangle$. In addition, assume that $X$ is a left-invariant vector field, i.e. $X\in {\mathfrak g}$. Since $\langle\cdot,\cdot\rangle$ is a left-invariant Riemannian metric, for an orthonormal basis relative to $\langle\cdot,\cdot\rangle$, we have
\begin{equation}\label{killing}
  \mathrm{Ric}=\lambda \mathrm{Id}-\frac{1}{2}[(\mathrm{ad} X)+(\mathrm{ad} X)^t]+\frac{|X|^2}{m}\mathrm{Pr}|_X.
\end{equation}
Here $\mathrm{Pr}|_X$ is the orthogonal projection onto $\mathbb{R}X$.
\begin{lemma}[\cite{CZ13}]\label{killingfield1}
Let $G$ be a unimodular Lie group with a left-invariant Riemannian metric $\langle\cdot,\cdot\rangle$ and the Lie algebra $\mathfrak g$. If $X\in{\mathfrak g}$ and
$\mathrm{Ric}^m_X=\lambda\langle\cdot,\cdot\rangle$, then $X$ is a Killing field.
\end{lemma}
Since a quadratic Lie group is unimodular, we know that Theorem~\ref{killingfield} follows from Lemma \ref{killingfield1}.

\section{Ricci curvature on quadratic Lie groups}
Let $\langle\cdot,\cdot\rangle$ be a left-invariant Riemannian metric on $G$, i.e. $$\langle \nabla_XY,Z\rangle+\langle Y,\nabla_XZ\rangle=0 \text{ for any } X,Y,Z\in {\mathfrak g}.$$ In fact, the Levi-Civita connection corresponding to the left-invariant Riemannian metric $\langle\cdot,\cdot\rangle$ is determined by the following equation
\begin{equation}\label{nabla}
\langle \nabla_XY,Z\rangle=\frac{1}{2}\left\{\langle [X,Y],Z\rangle - \langle [Y,Z],X\rangle + \langle [Z,X],Y\rangle\right\}.
\end{equation}There is a linear map $D$ of $\mathfrak g$ satisfying $$(X,Y)=\langle X,\theta(Y)\rangle \text{ for any } X,Y\in {\mathfrak g}.$$ It follows that $\theta$ is invertible, and symmetric with respect to $\langle\cdot,\cdot\rangle$. Since $\langle\cdot,\cdot \rangle$ is positive definite and $(\cdot,\cdot )$ is symmetric,
by a result in linear algebra, we can choose an orthonormal basis $\{X_1,\cdots,X_n\}$ of $\mathfrak g$ with respect to $\langle\cdot,\cdot\rangle$ such that $$(X_i,X_j)=\lambda_i\delta_{ij},$$ for some $\lambda_i\in{\mathbb R}$. Then $\theta(X_i)=\lambda_i(X_i)$.

By the equation~(\ref{nabla}) and the ad-invariance of $(\cdot,\cdot)$, we have
\begin{eqnarray*}
\langle \nabla_XY,Z\rangle&=&\frac{1}{2}\left\{\langle [X,Y],Z\rangle - \langle [Y,Z],X\rangle + \langle [Z,X],Y\rangle\right\} \\
&=&\frac{1}{2}\{([X,Y],\theta^{-1}Z) - ([Y,Z],\theta^{-1}X) + ([Z,X],\theta^{-1}Y)\}\\
&=&\frac{1}{2}\{(\theta^{-1}[X,Y],Z) - ([\theta^{-1}X,Y],Z) + ([X,\theta^{-1}Y],Z)\}\\
&=&\frac{1}{2}\langle \theta^{-1}[X,Y] - [\theta^{-1}X,Y] + [X,\theta^{-1}Y],\theta Z \rangle.
\end{eqnarray*}
It follows that
\begin{equation}\label{connection}
\nabla_X Y = \frac{1}{2}\{[X,Y] - \theta[\theta^{-1}X,Y] + \theta[X,\theta^{-1}Y]\},
\end{equation}
that is,
$$\nabla_X = \frac{1}{2}\{\mathrm{ad} X - \theta\mathrm{ad} (\theta^{-1}X) + \theta\mathrm{ad} X\theta^{-1}\}.$$

One has a simple formula of the Ricci curvature on a compact semisimple Lie group with respect to a left-invariant metric which was derived in \cite{Sa1}, and a simpler proof is given in \cite{DZ1}. The proof given in \cite{DZ1} is easily extended to a quadratic Lie group. That is,

\begin{lemma}\label{thm}
For any $X,Y\in{\mathfrak g}$, $\mathrm{Ric}(X,Y)=-\mathrm{tr}(\nabla_X-\mathrm{ad} X)(\nabla_Y-\mathrm{ad} Y)$.
\end{lemma}
\begin{proof} For any $X_i \in {\mathfrak g}$, we have
\begin{eqnarray*}
\mathrm{Ric}(X,Y)&=&\mathrm{tr}(X_i\mapsto\nabla_{X_i}\nabla_{X}Y-\nabla_{X}\nabla_{X_i}Y-\nabla_{[X_i,X]}Y)\\
&=&\mathrm{tr}(\nabla_{\nabla_XY}-\mathrm{ad}(\nabla_XY)-\nabla_X\nabla_Y+\nabla_X\mathrm{ad} Y+\nabla_Y\mathrm{ad} X-\mathrm{ad} Y\mathrm{ad} X)\\
&=&\mathrm{tr}(\nabla_{\nabla_XY})-\mathrm{tr}(\mathrm{ad}(\nabla_XY))-\mathrm{tr}(\nabla_X-\mathrm{ad} X)(\nabla_Y-\mathrm{ad} Y).
\end{eqnarray*}

By the left-invariance of $\langle\cdot,\cdot\rangle$, we know that $\nabla_x$ is skew-symmetric with respect to $\langle\cdot,\cdot\rangle$. It follows that $\mathrm{tr}\nabla_Z=0$ for any $Z\in{\mathfrak g}$. In particular, $\mathrm{tr}\nabla_{\nabla_XY}=0$. Since $G$ is quadratic, we have $\mathrm{tr}\mathrm{ad} X=0$ for any $X\in{\mathfrak g}$. In particular, $\mathrm{tr}(\mathrm{ad}(\nabla_XY))=0$. Then the theorem follows.\end{proof}

Assume that $\{X_1,\ldots,X_n\}$ is the above orthonormal basis of ${\mathfrak g}$ with respect to $\langle\cdot,\cdot\rangle$, and $C_{ij}^k$ are the structure constants with respect to the basis. That is,
$$[X_i,X_j]=\sum_{l=1}^n C_{ij}^lX_l.$$
By the ad-invariance of $(\cdot,\cdot)$, we have
\begin{equation}\label{equ1}C_{ij}^l\lambda_l=C_{jl}^i\lambda_i=C_{li}^j\lambda_j.\end{equation}
Let $\mu_i=\frac{1}{\lambda_i}$. Then we have
\begin{eqnarray*}
\nabla_{X_i}X_j&=&\frac{1}{2}([X_i,X_j]-\theta[\theta^{-1}X_i,X_j] + \theta[X_i,\theta^{-1}X_j])\\
&=&\frac{1}{2}(\mathrm{Id}-\mu_i\theta+\mu_j\theta)([X_i,X_j])\\
&=&\frac{1}{2}\sum_{l=1}^n\frac{\mu_l-\mu_i+\mu_j}{\mu_l}C_{ij}^lX_l.
\end{eqnarray*}
Then by Lemma~\ref{thm} and the equation~(\ref{equ1}), we have
\begin{eqnarray*}
\mathrm{Ric}(X_j,X_k)&=&-\mathrm{tr}(\nabla_{X_j}-\mathrm{ad} X_j)(\nabla_{X_k}-\mathrm{ad} X_k)\\
&=&-\sum_{i=1}^n\langle(\nabla_{X_j}-\mathrm{ad} X_j)(\nabla_{X_k}-\mathrm{ad} X_k)X_i,X_i\rangle\\
&=&-\sum_{i=1}^n\langle\{(\frac{1}{2}\sum_{l=1}^n\frac{\mu_l-\mu_k+\mu_i}{\mu_l}-1)C_{ki}^l\}(\nabla_{X_j}-\mathrm{ad} X_j)X_l,
X_i\rangle\\
&=&-\frac{1}{4}\sum_{i=1}^n\sum_{l=1}^n\frac{-\mu_l-\mu_k+\mu_i}{\mu_l}\frac{-\mu_i-\mu_j+\mu_l}{\mu_i}C_{ki}^l
C_{jl}^i \\
&=&\frac{1}{4}\sum_{i=1}^n\sum_{l=1}^n\frac{-\mu_l-\mu_k+\mu_i}{\mu_l}\frac{\mu_l-\mu_j-\mu_i}{\mu_l}C_{ki}^l
C_{ji}^l
\end{eqnarray*}
Furthermore, by the equation~(\ref{equ1}), we know
\begin{eqnarray*}
&&\frac{1}{4}\sum_{i>l}\frac{-\mu_l-\mu_k+\mu_i}{\mu_l}\frac{\mu_l-\mu_j-\mu_i}{\mu_l}C_{ki}^l
C_{ji}^l\\
&=&\frac{1}{4}\sum_{i<l}(-\mu_i-\mu_k+\mu_l)(\mu_i-\mu_j-\mu_l)\frac{C_{kl}^i}{\mu_i}\frac{C_{jl}^i}{\mu_i}\\
&=&\frac{1}{4}\sum_{i<l}(-\mu_i-\mu_k+\mu_l)(\mu_i-\mu_j-\mu_l)\frac{C_{ki}^l}{\mu_l}\frac{C_{ji}^l}{\mu_l}.
\end{eqnarray*}
Thus we have the following theorem.
\begin{theorem}\label{prop}
Let $\{X_1,\ldots,X_n\}$ be as above. Then
$$
\mathrm{Ric}(X_j,X_k)=-\frac{1}{2}\sum_{i<l}((\mu_l-\mu_i)^2-\mu_k\mu_j)\frac{C_{ki}^l}{\mu_l}\frac{C_{ji}^l}{\mu_l}.$$
\end{theorem}

\section{$m$-quasi-Einstein metrics on $G(n)$ for $m$ finite}\label{section4}
Let $G$ be a simply connected Lie group with the Lie algebra $\mathfrak g$, where $\{D,X,Y,Z\}$ is a basis of $\mathfrak g$ such that the non-zero brackets are given by
\begin{equation}
[D,X]=X, [D,Y]=-Y, [X,Y]=Z.
\end{equation}
It is easy to check that the symmetric bilinear form on $\mathfrak g$ satisfying
\begin{equation}
(D,Z)=(X,Y)
\end{equation}
is invariant. Thus $G$ is a quadratic Lie group if we take
\begin{equation}
(D,Z)=(X,Y)=\frac{1}{2}\text{ and other brackets equal to zero}.
\end{equation}
Let $e_1=D+Z$, $e_2=D-Z$, $e_3=X+Y$ and $e_4=X-Y$. Consider the left-invariant Riemannian metric $\langle\cdot,\cdot\rangle$ on $G$ defined by
\begin{equation}
\langle e_i,e_j\rangle=\delta_{ij}\lambda_i^2, \lambda_i\not=0, \text{ for any } i,j=1,2,3,4.
\end{equation}
Let $f_i=\frac{e_i}{\lambda_i}$ for any $1\leq i\leq 4$. Then we have
\begin{equation}
\langle f_i,f_j\rangle=\delta_{ij}, (f_1,f_1)=\frac{1}{\lambda_1^2},(f_2,f_2)=-\frac{1}{\lambda_2^2},(f_3,f_3)=\frac{1}{\lambda_3^2},(f_4,f_4)=-\frac{1}{\lambda_4^2},
\end{equation}
and the non-zero structure constants corresponding to the basis $\{f_i\}_{i=1,2,3,4}$ are
\begin{equation}
C_{13}^4=\frac{\lambda_4}{\lambda_1\lambda_3},C_{23}^4=\frac{\lambda_4}{\lambda_2\lambda_3},C_{14}^3=\frac{\lambda_3}{\lambda_1\lambda_4}, C_{24}^3=\frac{\lambda_3}{\lambda_2\lambda_4}, C_{34}^2=\frac{\lambda_2}{\lambda_3\lambda_4},C_{34}^1=-\frac{\lambda_1}{\lambda_3\lambda_4}.
\end{equation}
Any left-invariant Killing vector field with respect to $\langle\cdot,\cdot\rangle$ is of the form \begin{equation}a(\lambda_1f_1-\lambda_2f_2).\end{equation}
By Theorem~\ref{prop}, the Ricci curvatures are given by
\[ \left\{ \begin{aligned}
          & \mathrm{Ric}(f_1,f_3)=\mathrm{Ric}(f_1,f_4)=\mathrm{Ric}(f_2,f_3)=\mathrm{Ric}(f_2,f_4)=\mathrm{Ric}(f_3,f_4)=0,\\
          &\mathrm{Ric}(f_1,f_2)=-\frac{1}{2}((\lambda_3^2+\lambda_4^2)^2+\lambda_1^2\lambda_2^2)\frac{1}{\lambda_1\lambda_2\lambda_3^2\lambda_4^2},\\
          &\mathrm{Ric}(f_1,f_1)=-\frac{1}{2}((\lambda_3^2+\lambda_4^2)^2-\lambda_1^4)\frac{1}{\lambda_1^2\lambda_3^2\lambda_4^2},\\
          &\mathrm{Ric}(f_2,f_2)=-\frac{1}{2}((\lambda_3^2+\lambda_4^2)^2-\lambda_2^4)\frac{1}{\lambda_2^2\lambda_3^2\lambda_4^2},\\
          &\mathrm{Ric}(f_3,f_3)=-\frac{1}{2}((\lambda_1^2+\lambda_4^2)^2-\lambda_3^4)\frac{1}{\lambda_1^2\lambda_3^2\lambda_4^2}-
          \frac{1}{2}((\lambda_2^2-\lambda_4^2)^2-\lambda_3^4)\frac{1}{\lambda_2^2\lambda_3^2\lambda_4^2},\\
             &\mathrm{Ric}(f_4,f_4)=-\frac{1}{2}((\lambda_1^2-\lambda_3^2)^2-\lambda_4^4)\frac{1}{\lambda_1^2\lambda_3^2\lambda_4^2}-
          \frac{1}{2}((\lambda_2^2+\lambda_3^2)^2-\lambda_4^4)\frac{1}{\lambda_2^2\lambda_3^2\lambda_4^2}.
     \end{aligned}
   \right.
\]
Assume that $\langle\cdot,\cdot\rangle$ is $m$-quasi-Einstein for some $X\in {\mathfrak g}$. Then by Theorem~\ref{killingfield}, $X$ is a left-invariant Killing vector field with respect to $\langle\cdot,\cdot\rangle$. That is, \begin{equation}X=a(\lambda_1f_1-\lambda_2f_2) \text{ for some } a .\end{equation} Thus $\langle\cdot,\cdot\rangle$ is $m$-quasi-Einstein with the constant $\lambda$ and $a=1$ if and only if the following equations holds, i.e.
\[ \left\{ \begin{aligned}
          &-\frac{1}{2}((\lambda_3^2+\lambda_4^2)^2+\lambda_1^2\lambda_2^2)\frac{1}{\lambda_1\lambda_2\lambda_3^2\lambda_4^2}+\frac{\lambda_1\lambda_2}{m}=0,\\
          &-\frac{1}{2}((\lambda_3^2+\lambda_4^2)^2-\lambda_1^4)\frac{1}{\lambda_1^2\lambda_3^2\lambda_4^2}-\frac{\lambda_1^2}{m}=\lambda,\\
          &-\frac{1}{2}((\lambda_3^2+\lambda_4^2)^2-\lambda_2^4)\frac{1}{\lambda_2^2\lambda_3^2\lambda_4^2}-\frac{\lambda_2^2}{m}=\lambda,\\
          &-\frac{1}{2}((\lambda_1^2+\lambda_4^2)^2-\lambda_3^4)\frac{1}{\lambda_1^2\lambda_3^2\lambda_4^2}-
          \frac{1}{2}((\lambda_2^2-\lambda_4^2)^2-\lambda_3^4)\frac{1}{\lambda_2^2\lambda_3^2\lambda_4^2}=\lambda,\\
             &-\frac{1}{2}((\lambda_1^2-\lambda_3^2)^2-\lambda_4^4)\frac{1}{\lambda_1^2\lambda_3^2\lambda_4^2}-
          \frac{1}{2}((\lambda_2^2+\lambda_3^2)^2-\lambda_4^4)\frac{1}{\lambda_2^2\lambda_3^2\lambda_4^2}=\lambda.
     \end{aligned}
   \right.
\]
By the first equation, we can represent $m$ by $\{\lambda_i\}_{i=1,2,3,4}$. Put the representation of $m$ into the second and third equations, and then eliminate $\lambda$. Finally we have
\begin{equation}\label{lambda}\lambda_3^2=\lambda_4^2\text{ and }\lambda_1^2\lambda_2^2=4\lambda_4^4.\end{equation}
It is easy to check that $\{\lambda_i\}_{i=1,2,3,4}$ satisfying the equations~(\ref{lambda}) are the solutions of the above equations. On the other hand, non-zero $\{\lambda_i\}_{i=1,2,3,4}$ satisfying the equations~(\ref{lambda}) give an $m$-quasi-Einstein metric on $G$.

For the above case, $[\mathfrak g,\mathfrak g]$ is the Heisenberg Lie algebra of dimension 3. Motivated by the above example, we try to give examples when $[\mathfrak g,\mathfrak g]$ is the Heisenberg Lie algebra of higher dimension.

Let $G(n)$ be a simply connected Lie group with the Lie algebra ${\mathfrak g}(n)$, where $\{D,X_s,Y_s,Z\}_{s=1,\ldots,n}$ is a basis of ${\mathfrak g}(n)$ such that non-zero brackets are given by
\begin{equation}
[D,X_s]=a_sX_s, [D,Y_s]=-a_sY_s, [X_s,Y_s]=Z.
\end{equation}
We can assume that $0<a_1\leq a_2\leq\ldots\leq a_n$ by adjusting the order and the sign of the basis if necessary. It is easy to check that the symmetric bilinear form on ${\mathfrak g}(n)$ satisfying
\begin{equation}
(D,Z)=a_s(X_s,Y_s) \text{ for any } s=1,2,\ldots,n
\end{equation}
is ad-invariant. Thus $G$ is a quadratic Lie group if we take
\begin{equation}
(D,Z)=a_s(X_s,Y_s)=\frac{1}{2}\text{ for any } s \text{ and other brackets equal to zero}.
\end{equation}
Let $e_1=D+Z$, $e_2=D-Z$, $e_{2s+1}=\sqrt{a_s}(X_s+Y_s)$ and $e_{2s+2}=\sqrt{a_s}(X_s-Y_s)$ for any $1\leq s\leq n$. Consider the left-invariant Riemannian metric $\langle\cdot,\cdot\rangle$ on $G(n)$ defined by
\begin{equation}
\langle e_i,e_j\rangle=\delta_{ij}\lambda_i^2, \lambda_i\not=0, \text{ for any } i,j=1,2,\ldots,2n+2.
\end{equation}
Let $f_i=\frac{e_i}{\lambda_i}$ for any $1\leq i\leq 2n+2$. Then we have
\begin{equation}
\langle f_i,f_j\rangle=\delta_{ij}, (f_{2s+1},f_{2s+1})=\frac{1}{\lambda_{2s+1}^2},(f_{2s+2},f_{2s+2})=-\frac{1}{\lambda_{2s+2}^2}, \forall 0\leq s\leq n
\end{equation}
and the non-zero structure constants corresponding to the basis $\{f_i\}_{i=1,\ldots,2n+2}$ are
\begin{eqnarray}
&&C_{1(2s+1)}^{2s+2}=\frac{a_s\lambda_{2s+2}}{\lambda_1\lambda_{2s+1}}, C_{2(2s+1)}^{2s+2}=\frac{a_s\lambda_{2s+2}}{\lambda_2\lambda_{2s+1}}\\
&&C_{1(2s+2)}^{2s+1}=\frac{a_s\lambda_{2s+1}}{\lambda_1\lambda_{(2s+2)}}, C_{2(2s+2)}^{2s+1}=\frac{a_s\lambda_{2s+1}}{\lambda_2\lambda_{(2s+2)}} \\ &&C_{(2s+1)(2s+2)}^2=\frac{a_s\lambda_2}{\lambda_{2s+1}\lambda_{2s+2}},C_{(2s+1)(2s+2)}^1=-\frac{a_s\lambda_1}{\lambda_{2s+1}\lambda_{2s+2}}.
\end{eqnarray}
Any left-invariant Killing vector field with respect to $\langle\cdot,\cdot\rangle$ is of the form \begin{equation}a(\lambda_1f_1-\lambda_2f_2).\end{equation}
By Theorem~\ref{prop}, $\mathrm{Ric}(f_i,f_j)=0$ except $i=j$ or $(i,j)=(1,2)$. Furthermore,
\[ \left\{ \begin{aligned}
          &\mathrm{Ric}(f_1,f_2)=\sum_{s=1}^n-\frac{1}{2}((\lambda_{2s+1}^2+\lambda_{2s+2}^2)^2+\lambda_1^2\lambda_2^2)\frac{a_s^2}{\lambda_1\lambda_2\lambda_{2s+1}^2\lambda_{2s+2}^2},\\
          &\mathrm{Ric}(f_1,f_1)=\sum_{s=1}^n-\frac{1}{2}((\lambda_{2s+1}^2+\lambda_{2s+2}^2)^2-\lambda_1^4)\frac{a_s^2}{\lambda_1^2\lambda_{2s+1}^2\lambda_{2s+2}^2},\\
          &\mathrm{Ric}(f_2,f_2)=\sum_{s=1}^n-\frac{1}{2}((\lambda_{2s+1}^2+\lambda_{2s+2}^2)^2-\lambda_2^4)\frac{a_s^2}{\lambda_2^2\lambda_{2s+1}^2\lambda_{2s+2}^2},
              \end{aligned}
   \right.
\]
and for any $1\leq s\leq n$,
\[ \left\{ \begin{aligned}
          &\mathrm{Ric}(f_{2s+1},f_{2s+1})=-\frac{1}{2}((\lambda_1^2+\lambda_{2s+2}^2)^2-\lambda_{2s+1}^4)\frac{a_s^2}{\lambda_1^2\lambda_{2s+1}^2\lambda_{2s+2}^2} \\
          & \quad\quad\quad\quad\quad\quad\quad\quad -\frac{1}{2}((\lambda_2^2-\lambda_{2s+2}^2)^2-\lambda_{2s+1}^4)\frac{a_s^2}{\lambda_2^2\lambda_{2s+1}^2\lambda_{2s+2}^2},\\
             &\mathrm{Ric}(f_{2s+2},f_{2s+2})=-\frac{1}{2}((\lambda_1^2-\lambda_{2s+1}^2)^2-\lambda_{2s+2}^4)\frac{a_s^2}{\lambda_1^2\lambda_{2s+1}^2\lambda_{2s+2}^2}   \\    & \quad\quad\quad\quad\quad\quad\quad\quad -\frac{1}{2}((\lambda_2^2+\lambda_{2s+1}^2)^2-\lambda_{2s+2}^4)\frac{a_s^2}{\lambda_2^2\lambda_{2s+1}^2\lambda_{2s+2}^2}.
     \end{aligned}
   \right.
\]
Assume that $\langle\cdot,\cdot\rangle$ is $m$-quasi-Einstein for some $X\in {\mathfrak g}(n)$. By Theorem~\ref{killingfield}, $X$ is a left-invariant Killing vector field with respect to $\langle\cdot,\cdot\rangle$. That is, \begin{equation}X=a(\lambda_1f_1-\lambda_2f_2) \text{ for some } a .\end{equation} Assume that $\langle\cdot,\cdot\rangle$ is $m$-quasi-Einstein with the constant $\lambda$ and $a=1$.
Then for any $1\leq s\leq n$, $\mathrm{Ric}(f_{2s+1},f_{2s+1})=\mathrm{Ric}(f_{2s+2},f_{2s+2})$. It follows that
\begin{equation}\label{1}\lambda_{2s+1}^2=\lambda_{2s+2}^2.\end{equation} For any $1\leq i\not=j\leq n$, since $\mathrm{Ric}(f_{2i+1},f_{2i+1})=\mathrm{Ric}(f_{2j+1},f_{2j+1})$, we have
\begin{equation}\label{2}\frac{\lambda_{2i+1}^2}{a_i}=\frac{\lambda_{2j+1}^2}{a_j}.\end{equation}
By the first equation, we can represent $m$ by $\{\lambda_i\}_{i=1,\ldots,2n+2}$, and then put the representation of $m$ into $\mathrm{Ric}^m_X(f_{1},f_{1})$ and $\mathrm{Ric}^m_X(f_{2},f_{2})$. It is easy to check that $$\mathrm{Ric}^m_X(f_{1},f_{1})=\mathrm{Ric}^m_X(f_{2},f_{2}).$$ Since $\mathrm{Ric}^m_X(f_{1},f_{1})=\mathrm{Ric}^m_X(f_{3},f_{3})$, we have
\begin{equation}\label{3}
\lambda_1^2\lambda_2^2=\frac{4{\sum_{s=1}^na_s^2}}{a_1^2}\lambda_3^4.
\end{equation}
It is easy to check that $\{\lambda_i\}_{i=1,\ldots,2n+2}$ satisfying the equations~(\ref{1}), (\ref{2}) and (\ref{3}) are the solutions. On the other hand, non-zero $\{\lambda_i\}_{i=1,\ldots,2n+2}$ satisfying the equations~(\ref{1}), (\ref{2}) and (\ref{3}) give an $m$-quasi-Einstein metric on $G(n)$. It is easy to check that
\begin{equation}
S=\sum_{i}\mathrm{Ric}(f_i,f_i)=-n(\lambda^2_1+\lambda^2_2)(\frac{2}{\lambda^2_1\lambda^2_2}+\frac{a_1^2}{2\lambda_3^4}).
\end{equation}
With respect to the orthonormal basis $\{f_i\}_{i=1,\ldots,2n+2}$, the determinant of the metric matrix with respect to $\langle\cdot,\cdot\rangle$ is $1$. Take $\lambda_3=c\not=0$. Then by the equation~(\ref{3}), $\lambda_1^2\lambda_2^2$ is a constant. It follows that $S$ is based on the choice of $\lambda_1^2$ and $\lambda_2^2$. That is, $G(n)$ admits infinitely many non-equivalent non-trivial $m$-quasi-Einstein metrics for $m$ finite, i.e. Theorem~\ref{G(n)} holds.

\section{Acknowledgments}
This work is supported by NSFC (No. 11001133, No. 11271143, No. 11371155),  and University Special Research
 Fund for Ph.D. Program of China (20124407110001 and 20114407120008). We would like to thank F.H. Zhu for the helpful comments, conversation and suggestions.

\end{document}